\documentclass[english]{amsart}
\usepackage[T1]{fontenc}
\usepackage{enumitem}
\usepackage{euscript}
\usepackage{textcomp}
\usepackage{subfig}
\usepackage{amsmath}
\usepackage{amssymb}
\usepackage{amsthm}
\usepackage[english]{babel}
\usepackage{graphicx}
\theoremstyle{definition}
\newtheorem{definition}{Definition}
\newtheorem{theorem}{Theorem}
\newtheorem{lemma}[theorem]{Lemma}
\newtheorem{proposition}[theorem]{Proposition}

\newtheorem*{remark}{Remark}

\begin{document}
\title[Cohomological equations for linear involutions]{Cohomological equations for linear involutions}

\date{\today}
\author{Erwan Lanneau, Stefano Marmi and Alexandra Skripchenko}

\address{Institut Fourier, Universit\'e Grenoble-Alpes, BP 74, 38402 Saint-Martin-d'H\`eres, France}
\email{erwan.lanneau@univ-grenoble-alpes.fr}
\address{Scuola Normale Superiore, Piazza dei Cavalieri 7, 56126 Pisa, Italy}
\email{stefano.marmi@sns.it}
\address{Faculty of Mathematics, National Research University Higher School of Economics, Usacheva St. 6, 119048 Moscow, Russia \textit{and}
Skolkovo Institute for Science and Technology, Skolkovo Innovation Center, 143026 Moscow, Russia}
\email{sashaskrip@gmail.com}

\begin{abstract}
In the current note we extend results by Marmi, Moussa and Yoccoz about cohomological equations for interval exchange transformations to irreducible linear involutions.
\end{abstract}
\maketitle

\section{Historical background and statement of the result} 

\subsection{Historical background}
Cohomological equations in connection with interval exchange transformations and linear flows were widely studied in the last 20 years. 
Historically the first one was a pioneer work by Forni (\cite{Fo1}) on the cohomological equation associated to linear flows on surfaces of higher genus. His interest toward cohomological equations was
motivated by the studying of time changing of the flows. He developed some advanced analytical tools to prove the existence of the solution for cohomological equations for almost all direction of the linear flow for a given translation surface. 

In the current note we generalize one of the first and the most important results for cohomological equations for interval exchange transformation (\cite{MaMoYo1}) to irreducible linear involutions. With respect to Forni's paper, the result of~\cite{MaMoYo1} is sharper (from the point of view of degree of smoothness) and an explicit Diophantine condition (called Roth type) is established; the proof itself relies on the detailed analysis of the  properties of the renormalization map (Rauzy induction). This result was further refined and generalized in a several directions. 

First of all, the proof of the full measure statement was significantly simplified by C. Matheus in~\cite{MaMaMo}, where the author used some results established in~\cite{AGY} in connection with the exponential tail of the roof function for Teichm\"uller flow to rewrite the hardest technical part, i. e. proof of the full measure statement for the condition (a) (see below). 
 
Also, in~\cite{MaYo} a further development of the ideas from~\cite{MaMoYo1} leads to an improvement of the regularity part: the solutions of the cohomological equation for \emph {restricted Roth type} 
interval exchange maps (see more details on this notion below) are H\"older continuous provided that the datum is of class $C^{r}$ with $r > 1$ and belongs to a finite-codimension linear subspace.

The notion of the Roth type IET was slightly modified in~\cite{MaMoYo2} and adapted to the question of the conjugacy of generalized (affine) interval exchange transformations and obstructions for the linearization. The set of considered IETs is called \emph{restricted Roth type} and has a full Lebesgue measure in the parameter space. For this class, it was shown that among $C^{r+3}$-deformation of any IET $T_0$ from this class that are $C^{r+3}$ tangent to $T_0$ at the singularities, those that are conjugated to $T_0$ by a $C^{r}$-diffeomorphism ($r\ge 2$) close to the identity form a $C^1$-submanifold of codimension $(g-1)(2r+1) + s$. 

In~\cite{FoMaMa} statements from~\cite{MaMoYo2} are extended for some particular classes of self-similar IETs, including those that are associated to the Eierlegende Wollmilchsau and the Ornithorynque. 

In~\cite{MaUlYo} a new type of Diophantine condition is introduced:  the absolute Roth type condition is a weakening of the notion of Roth type. It is shown that the results from~\cite{MaYo} can be extended for restricted absolute Roth type IETs. Also, from the work by Chaika and Eskin on Oseledets genericity of the Kontsevich-Zorich cocycle (see~\cite{ChEs}) one gets that for a given translation surface for almost every direction of the linear flow on it the first return map to the transversal satisfies the absolute Roth type condition.

A completely new approach to the solution of the cohomological equation in the special case of pseudo-Anosov maps, employing transfer operator technique,s has been successfully developed in \cite{FaGoLa}.

Finally, Zubov in~\cite{Zu} treated the question of existence of the solution of cohomological equation from the point of view of symbolic dynamics: he worked in a framework of suspension flows over automorphisms of Markov compacta and established the sufficient condition of the existence of the bounded solution of cohomological equation in terms of finitely additive measures for Markov compacta. 

\subsection{Statement of the result} 

Our results represent another way to generalize~\cite{MaMoYo1}: we replace IETs by \emph{linear involutions}, which appear to be the first return map for the foliations, defined by quadratic differentials. Danthony and Nogueira introduced linear involutions in~\cite{DaNo}; they also described Rauzy induction for linear involutions. Linear involutions are encoded by a combinatorial datum (known as \emph{generalized permutation}) and by continuous data (lengths of the intervals $\lambda_i$). More precisely, let $\mathcal{A}$ be an alphabet on $d$ letters; then a generalized permutation of a type $(l,m)$ with $l+m=2d$ is a two-to-one map $\pi: \{1,\cdots, 2d\}\to\mathcal{A}$ (we also call the letters of the alphabet \emph{names} sometimes); as for the lengths, there is a natural condition $\sum_{i=1}^{l}\lambda_{\pi(i)} = \sum_{i=l+1}^{l+m}\lambda_{\pi(i)}$. 

Boissy and Lanneau in~\cite{BoLa} defined a special class of geometrically meaningful (also known as \emph{irreducible}) generalized permutations, namely for which that there exists a linear involution associated to this generalized permutation, such that this linear involution represents an appropriate cross section of the vertical foliation of some quadratic differential. They also studied the relationship between the connected components of the strata of the moduli spaces of quadratic differentials and [extended] Rauzy classes of irreducible generalized permutations (see~\cite{BoLa} for details). Let $T$ denote a linear involution defined on a set of $2d$ intervals $\sqcup A_i$. If the underlying permutation is irreducible (see \cite{BoLa} for precise definition) then there is a quadratic differential associated to it, and we will say that $T$ \emph{belongs to} the corresponding stratum of the moduli space. Moreover the Rauzy - Veech induction is well defined and described by matrices of size $d$ and the sequence of arrows in the Rauzy diagram $\mathcal{D}$; the arrows are named by the winner of each step. Zorich acceleration of this algorithm and the further acceleration introduced in ~\cite{MaMoYo1} can be defined exactly in the same way as for classical interval exchange transformations. We will denote by $Z(k)$ the product of Rauzy matrices that correspond to a concatenation of paths in the Rauzy diagram $\mathcal{D}$ such that all letters but one win at least once. We will also denote by $T(k)$ the linear involution obtained by iterating the accelerated Rauzy-Veech induction; its datum is defined by $(\pi(k), \lambda(k))$. 

We will also use the notation
$$
Q(k)=Z(1)\cdots Z(k)
$$
(see~\cite{BoLa} and~\cite[Section 1.2.4]{MaMoYo1} for details).

The Rauzy-Veech induction can also be defined on the level of an oriented double cover (as in~\cite{AvRe,Gu}).
More precisely it corresponds to the ``\emph{minus}'' Rauzy-Veech groups of~\cite[Section 4]{Gu}.
In order to state our result, we will define functions that are anti-invariant with respect to the involution on the double cover, or equivalently, functions that agree with the given irreducible linear involution (which implies in particular that for a given linear involution it is enough to define such a function on $d$ intervals). 

Let $\widehat{BV}_{*}^1(\sqcup A_i)$ be the set of functions $\widehat\Phi$ of class $C^1$ defined on each interval of continuity and with a derivative of bounded variation and whose integral on the domain vanishes. We also require that each $\widehat \Phi$ agrees with the given irreducible linear involution (the values of $\widehat \Phi$ on subintervals with the same name coincide or are the opposite depending on the involution). The definition of the Roth type interval exchange transformation given in~\cite{MaMoYo1} carries over the case of linear involutions (see Section~\ref{Roth} for the details). Then we have the following 
\begin{theorem}\label{A}
Let $T$ be a minimal linear involution of Roth type that does not belong to the following strata: $\mathcal{Q}(4g-4)$ (minimal),
or $\mathcal{Q}(2a, 2b, 2c,\cdots,2z)$. Then for any $\widehat\Phi\in \widehat{BV}_{*}^1(\sqcup A_i)$ there exists a function $\chi$ constant on each interval $A_i$ and that agrees with the given irreducible linear involution $T$, and a bounded function $\Psi$ such that  $$\Psi-\Psi\circ T=\Phi-\chi.$$
\end{theorem}

\begin{remark}
The restriction concerning the strata is rather technical: it comes from the fact that we use the result about simplicity of spectrum for certain stratum of quadratic differentials obtained by Guti\'errez-Romo in~\cite{Gu} (see section \ref{bc} for the details). Recently Bell, Delecroix, Gadre, Gutierrez-Romo and Schleimer informed us that it should be possible to extend the main result of~\cite{Gu} to all strata by using veering triangulations. In a such a case Theorem~\ref{A} will be automatically extended to all strata.
\end{remark}
\begin{remark}
Theorem \ref{A} above states that $\Psi$ and $\chi$ must agree with the linear involution. This rules out the possibility of deducing the statement by direct application of the statement proved in \cite{MaMoYo1} or \cite{MaMaMo} to IET appearing on a double cover level since a priori $\chi$ obtained in this way does not agree with the given irreducible linear involution.
\end{remark}

We prove also

\begin{theorem}
\label{B}
Roth-type linear involutions form a full measure set in the space of all irreducible linear involutions. 
\end{theorem}

\begin{remark}
It is important to mention that the approach used in~\cite{MaMaMo} is not applicable in our case: despite the fact that the result similar to the main statement of~\cite{AGY} for quadratic differentials was established in~\cite{AvRe}, the proof is based on somehow weaker estimations that can not be used to replace the construction introduced in~\cite{MaMoYo1}. 
\end{remark}

The note is organized in the following way: in Section~\ref{Roth} we define Roth-type linear involutions; in Section~\ref{equation} we describe the changes that are required to adapt the proof of Theorem~\ref{A} in~\cite{MaMoYo1} to the case of linear involutions; in Section~\ref{measure} we prove Theorem~\ref{B} (once again, emphasizing only the steps of the proof that are substantially different from the one presented in~\cite{MaMoYo1}). 

 \section{Roth-type linear involutions}\label{Roth}
Essentially, our definition of Roth-type maps coincides with the one in~\cite{MaMoYo1} (up to some technical differences between IETs and linear involutions). Here and in what follows given a matrix $A$ we denote by $||A|| = \max|A_{\alpha, \beta}|.$ 

Namely, we introduce three key conditions:
\begin{enumerate}[label=\alph*)]
\item \emph{Growth-rate condition:}
for every $\epsilon>0$ there exists $C_{\epsilon}>0$ such that for all $k\ge0$ we have 
\begin{equation}\label{condition1}
||Z(k+1)||\le C_{\epsilon}||Q(k)||^{\epsilon}
\end{equation}
This condition can be seen as a kind of a bounded distortion condition for the matrices appearing in an accelerated version of the Zorich continued fraction algorithm. Moreover, it reduces to the standard notion of Roth type irrational numbers in case of rotations (IETs on 2 intervals). 

\item \emph{Spectral gap}: the second condition is a spectral condition which guarantees unique ergodicity of Roth type minimal linear involutions. Formally, the condition is defined as follows:
for each $k\ge 0$, let $\Gamma^{(k)}$ be a copy of $\mathbb{R}^{d}$ (we consider it as a space of functions that are constant on each interval of continuity of our linear involution and agree with this linear involution). For $0\le k\le l$, let $S(k,l)$ be the linear map from $\Gamma^{(k)}$ to $\Gamma^{(l)}$ whose matrix in the
canonical basis is $^tQ(k,l)$. This can be interpreted as a
special Birkhoff sum (see~\cite[Section 2]{MaMoYo1} for details). 

As it was mentioned above, linear involutions are encoded by generalized permutations; each generalized permutation consists of the letters of some alphabet and every letter appears there twice. 

For $\phi =(\phi_\alpha )_{\alpha\in\hat{\mathcal{A}}}\in\Gamma^{(k)},$ define $$ I_k(\phi
)=\sum_{\alpha\in\hat{\mathcal{A}}}\lambda_\alpha^{(k)}\phi_\alpha\; ; $$ 
We then have $$ I_l(S(k,l)\phi)=I_k(\phi )\; . $$ Denote by
$\Gamma^{(k)}_*$ the kernel of the linear form $I_k$. We will ask
the following:

{\it There exists $\theta >0, C>0$ such that, for all
$k\ge 0$, we have} $$ \Vert S(k)\mid_{\Gamma^{(0)}_*}\Vert\le
C\Vert S(k)\Vert^{1-\theta}=C\Vert Q(k)\Vert^{1-\theta}\; . $$

\item \emph{Coherence}:
to define our third condition, we
consider again the operators $S(k,l)\, :\,
\Gamma^{(k)}\rightarrow\Gamma^{(l)}$. Let $\Gamma^{(k)}_s$ be the
linear subspace of $\Gamma^{(k)}$ whose elements $v$ satisfy the
following: there exists $\sigma =\sigma (v)>0$, $C=C(v)>0$ such
that, for all $l\ge k$, one has $$\Vert S(k,l)v\Vert\le C\Vert
S(k,l)\Vert^{-\sigma}\Vert v\Vert\; . $$

We call $\Gamma^{(k)}_s$ the stable subspace of $\Gamma^{(k)}$.
Obviously, one has $\Gamma^{(k)}_s\subset\Gamma^{(k)}_*$. Since the linear involution has the form $T(x) = \pm x + \delta_i,$ where $\delta_i$ is an analogue of translation vector in case of IETs, one can easily show that 
$\Gamma^{(k)}_s$ is never reduced to $0$ because it
always contains the translation vector
$(\delta_\alpha^{(k)})_{\alpha\in\mathcal{A}}$ (this statement can be manually verified using invariant antisymmetric form associated with linear involutions defined in \cite{AvRe};  see also \cite{Gu}). 

The operator $S(k,l)$ maps $\Gamma^{(k)}_s$ onto $\Gamma^{(l)}_s$.
Therefore we can define a quotient operator 
$$ S_{(k,l)}:
\Gamma^{(k)}/\Gamma^{(k)}_s\rightarrow
\Gamma^{(l)}/\Gamma^{(l)}_s\; . $$ 

The condition (c) states that the norm
of the inverse of $S_{(k,l)}$ is not too large:
{\it for any $\varepsilon >0$, there exists
$C_\varepsilon >0$ such that, for all $l\ge k$, we have}
$$ \Vert [S_{(k,l)}]^{-1}\Vert \le
C_\varepsilon\Vert Q(l)\Vert^\varepsilon\; $$
$$ \Vert S
(k,l)\mid_{\Gamma^{(k)}_s}\Vert \le C_\varepsilon\Vert
Q(l)\Vert^\varepsilon\;$$

\end{enumerate}

\subsection{Conditions (b) and (c) in terms of Lyapunov exponents}\label{bc}
As it was described in~\cite[Section 4.5]{MaMoYo1}, the cocycle that corresponds to Zorich acceleration is log-integrable with respect to the ergodic Masur-Veech measure and therefore Lyapunov exponents are defined almost everywhere. The same property holds for the case of linear involutions (see, for example,~\cite{Gu} and much stronger estimations in~\cite{AvRe}). 

Then in case of Abelian differentials the space $\Gamma_s$ was associated to the negative Lyapunov exponents. The same logic holds for the case of quadratic differentials with one major modification: in non-orientable case there are two natural families of Lyapunov exponents instead of one; usually these groups are called \emph{plus} and \emph{minus}, and the Lyapunov exponents from these groups are responsible for different ergodic properties. The formal definitions of plus and minus Lyapunov exponents can be found in~\cite{Tre,Gu}. It follows from the definition that 

\begin{lemma}
\label{LM:b}
Condition (b) holds almost always for linear involutions associated with quadratic differentials that do not belong to the following strata: $\mathcal{Q}(4g-4)$ (minimal)
or $\mathcal{Q}(2a, 2b, 2c,\cdots,2z)$.
\end{lemma}

\begin{proof}
Since in our case $\Gamma_s$ is associated with the negative \emph{plus} Lyapunov exponents, the lemma follows from~\cite[Theorem 1.1]{Gu}. 
\end{proof} 

Also, as in \cite{MaMoYo1} Oseledets theorem implies 
\begin{lemma}
\label{LM:c}
Coherence property follows for almost all irreducible linear involutions. 
\end{lemma}

\section{Cohomological equation}\label{equation}
Morally, the proof of Theorem A in our case coincides with the one suggested in~\cite{MaMoYo1} and contains a few steps: 
\begin{itemize}
\item Gottshalk-Hedlund theorem implies that in case of a minimal homeomorphism of a compact space the question about existence of the solution of cohomological equations can be reduced to the question about the boundedness of Birkhoff sums. As in~\cite{MaMoYo1}, our map is minimal but not continuous; so we need to apply some kind of Denjoy-style construction which is slightly different to to the one suggested in~\cite{MaMoYo1}. We provide the construction in Section~\ref{denjoy}.
\item  the next step in the proof is the reduction of the control of a general Birkhoff sum to the control of those \emph{special} Birkhoff sums which are obtained by considering the return times of the point under iteration of the map. This step in our case works exactly in the same way as in~\cite{MaMoYo1}. 
\item the estimates of these special Birkhoff sums for functions of bounded variation result in the proof of the statement. At this moment we apply properties (a), (b) and (c) and use two equivalent statement of property (a) (the one in terms of matrices and their norms and the one in terms of lengths of intervals). The proof of the equivalence of these two statements is slightly different in our case; the detailed proof is provided in Section~\ref{size}.
\end{itemize}

\subsection{Size of matrices}\label{size}
We can reformulate (a) in terms of the lengths $\lambda_\alpha^{(k)}$. We use as norm of a
matrix the sum of all coefficients.

As in~\cite{MaMoYo1}, we have the following
\begin{proposition}

For $k\ge 0$ 
$$\max_{\alpha\in \mathcal{A}}\lambda_\alpha^{(k)}\ge
\lambda^*\Vert
Q(k)\Vert^{-1}\ge \min_{\alpha\in \mathcal{A}}\lambda_\alpha^{(k)}\; .
$$

Condition (a) is equivalent to the following converse
estimate: for all $\varepsilon > 0$, there exists $C_\varepsilon
>0$ such that 
\begin{equation}\label{propA}
\max_{\alpha\in \mathcal{A}}\lambda_\alpha^{(k)}\le
C_\varepsilon \min_{\alpha\in \mathcal{A}}\lambda_\alpha^{(k)}\Vert
Q(k)\Vert^{\varepsilon}\; . 
\end{equation}
\end{proposition} 
\begin{proof}
The proof for the fact that condition (a) implies~\ref{propA} repeats verbatim the proof in~\cite{MaMoYo1}. 
Now let is show that if we have~\ref{propA}, condition (a) also holds. 
It follows from the definition of Zorich's acceleration that 
\begin{equation}\label{change1} \max_{\alpha\in \mathcal{A}}\lambda_\alpha^{(k)}\ge \min_{\alpha\in \mathcal{A}}\lambda_\alpha^{(k+1)}\frac{\Vert Z(k+1)\Vert} {2d}\end{equation}

Then, there exists $\alpha_0\in\mathcal{A}$ such that $\lambda_{\alpha_0}^{(k)}=\lambda_{\alpha_{0}}^{(k+1)}$.

Now, by the assumption of the proposition, we have that 
$$\lambda_{\alpha_0}^{(k+1)}\le \max_{\alpha\in \mathcal{A}}\lambda_\alpha^{(k+1)}\le C_\varepsilon \min_{\alpha\in \mathcal{A}}\lambda_\alpha^{(k+1)}\Vert
Q(k+1)\Vert^{\varepsilon}.$$

Our observation (\ref{change1}) implies that 
\begin{align*}
C_\varepsilon \min_{\alpha\in \mathcal{A}}\lambda_\alpha^{(k+1)}\Vert Q(k+1)\Vert^{\varepsilon}\le \\
2dC_\varepsilon {\Vert Z(k+1)\Vert}^{-1}\max_{\alpha\in \mathcal{A}}\lambda_\alpha^{(k)}\Vert
Q(k)\Vert^{\varepsilon}\Vert Q(k+1)\Vert^{\varepsilon} \le \\ 
2d{C_\varepsilon}^2{\Vert Z(k+1)\Vert}^{-1}\max_{\alpha\in \mathcal{A}}\lambda_\alpha^{(k)}\Vert
Q(k)\Vert^{\varepsilon}\Vert
Q(k+1)\Vert^{\varepsilon}\le \\
2d{C_\varepsilon}^2{\Vert Z(k+1)\Vert}^{-1}\lambda_\alpha^{(k)}\Vert
Q(k)\Vert^{\varepsilon}\Vert
Q(k+1)\Vert^{\varepsilon},
\end{align*}
therefore $$\Vert Z(k+1)\Vert\le 2d{C_\varepsilon}^2\Vert
Q(k)\Vert^{\varepsilon}\Vert
Q(k+1)\Vert^{\varepsilon},$$
so condition (a) holds.
\end{proof}

\subsection{From Gottschalk-Hedlund theorem to linear involutions} \label{denjoy}
In this section we follow general strategy of~\cite[Section 2.1.2]{MaMoYo1}. Irreducible linear involution of Keane's type is minimal but not continuous. So, as in~\cite{MaMoYo1}, we need some kind of Denjoy construction to bypass this issue. 
If our linear involution is given by a map $T=f\circ \widehat T$ where $\widehat T$ acts $\widehat X = X\times {0,1}$, $Sing$ is the set of singular points and $f$ is the involution: $f(x, \epsilon)=(x, 1-\epsilon)$, then 
\begin{equation}\label{change2}
D_0=\{{\widehat T}^{-n}(Sing)\}, D_1 = \{{\widehat T}^n(f(Sing))\}. 
\end{equation}

The rest of the proof repeats verbatim the construction described in~\cite{MaMoYo1}. 

\section{Full measure}\label{measure}
In this section we prove Theorem \ref{B}. Since we already checked above that Condition (b) and Condition (c) (Lemma \ref{LM:b} and Lemma \ref{LM:c}) hold almost everywhere, we will focus on Condition $(a)$.
\subsection{Condition (a)}
The main step of the proof parallels~\cite[Section 4.6]{MaMoYo1}. The key proposition is the following one: 
\begin{proposition}
\label{prop:comb}
There exist an integer $l = l(d)$ and a constant $\eta = \eta(d) > 0$ with the following properties. 
Let $\gamma = (\gamma(n))_{0<n\leq N}$ be a finite path in $\mathcal D$ such that the set $\mathcal A'$ of names of arrows of $\gamma$ is distinct from $\mathcal A$. Assume that $D = \mathrm{card}\ \mathcal A' > 1$. Then there exists a subset $\Delta'(\gamma) \subset \Delta(\gamma)$ with
$$
\mathrm{vol}_{d-1}(\Delta'(\gamma)) \geq \eta \mathrm{vol}_{d-1}(\Delta(\gamma))
$$
such that, for every $T \in \Delta'(\gamma)$, there exists $M > N$ with
\begin{enumerate}
\item the name of $\gamma^{(M)}(T)$ does not belong to $\mathcal A'$;
\item no more than $l (D -1)$-segments are needed to cover $(\gamma(n))_{N\leq n<M}$.
\end{enumerate}
\end{proposition}

The fact that the full measure estimate for Condition (a) is a corollary of Proposition~\ref{prop:comb} is obtained
by using a tricky argument (see~\cite[Section 4.7]{MaMoYo1}) and a standard Borel-Cantelli argument. 
We now will explain how to prove Proposition~\ref{prop:comb}.

\subsection{Proof of Proposition~\ref{prop:comb}}

To be consistant, we will use notations of~\cite[\S 4.8]{MaMoYo1}. We denote by $Q_\beta(\gamma)$ the sum of the column $\beta$ of the
matrix associated to the path $\gamma$. If $T$ is a linear involution, we denote by $Q_\alpha(n,T)=Q_\beta(\gamma)$ where $\gamma$ is the finite path up to step $n$.\\
Let $\gamma, 
\mathcal A', D$ be as in the assumption. Pick an irreducible linear involution $T\in \Delta(\gamma)$  satisfying Keane's condition. 
For any $n\geq 0$ we denote
\begin{equation}
\begin{array}{c}
Q'(n,T) = \sum_{\alpha \in \mathcal A'} Q_{\alpha}(n,t)\\ \\
Q_{\mathrm{ext}}(n,T) = \sum_{\alpha \in \mathcal A \backslash \mathcal A'} Q_{\alpha}(n,t)\\
\end{array}
\end{equation}

\begin{lemma}
\label{lm:combi:1}
If the names of the arrows $\gamma(m)(T)$ belong to $\mathcal A'$ for $m \leq n$, we have 
$$
Q_{\mathrm{ext}}(n, T ) \leq (2d - 5)Q' (n, T ).
$$
\end{lemma}
\begin{proof}[Proof of Lemma~\ref{lm:combi:1}]
We start by dividing the segment $[1,n]$ into maximal $1$-segments $I_i=[n_i,n_{i+1})$, for $i=0,\dots,k$, into which the name of the arrows is 
$\alpha_i$. By assumption $\alpha_i\in\mathcal A'$. For every $i\geq 0$, if the letter $\alpha_i$ is simple then
the secondary names of the arrows in the interval $I_i$ appear with some periodicity, say $d_i \leq 2d-4$.
Moreover if $i > 0$ then the secondary name of $\gamma^{(n_i)}$ is $\alpha_{i-1}$.

In the first segment $[1,n_1)$, if the letter $\alpha_0$ is simple, we write $n_1 = kd_0 + n'_1$ with $0 < n'_1 \leq d_0$
and the secondary name of $\gamma^{(m)}$ is $\alpha_{1}$ for each $m=n_1-kd_0$, $k>0$. We will write also
$n_1 = kd_0 + n'_1$, $0 < n'_1 \leq d_0$.

Now for each $m \in I_i$, if the secondary name of $\gamma(m)$ is $\beta_{m,i}$ then
\begin{equation}
\begin{array}{l}
Q_{\mathrm{ext}} (m, T ) = Q_{\mathrm{ext}} (m - 1, T )\\
Q'(m,T)=Q'(m-1,T)+Q_{\alpha_i}(n_i -1,T)
\end{array}
\qquad
\textrm{if } \beta_{m,i} \in \mathcal A',
\end{equation}
and
\begin{equation}
\begin{array}{l}
Q_{\mathrm{ext}} (m, T ) = Q_{\mathrm{ext}} (m - 1, T ) +Q_{\alpha_i}(n_i -1,T)\\
Q'(m,T)=Q'(m-1,T)
\end{array}
\qquad
\textrm{if } \beta_{m,i} \not \in \mathcal A'.
\end{equation}
We have to control when the secondary name does not belong to $\mathcal A'$. Obviously in every segments
$[n'_1,n_{1})$ and $[n_i,n_{i+1})$ for $i\geq 1$, if the name of the arrows is simple then
$$
\frac{\#\{\textrm{secondary names}\not \in \mathcal A' \}}{\#\{\textrm{secondary names} \in \mathcal A' \}} \leq \frac{d_i - 1}{1} \leq 2d-5
$$
and, if the name of the arrows is double then
$$
\frac{\#\{\textrm{secondary names}\not \in \mathcal A' \}}{\#\{\textrm{secondary names} \in \mathcal A' \}} \leq 2d-4
$$
Finally we have to estimates $Q_{\mathrm{ext}}$ and $Q'$  for $m\in [0,n'_{1})$:
$$
Q'(m,T) \geq D \geq 2, \textrm{ and}
$$
$$
Q_{\mathrm{ext}}(m,T) \leq  Q_{\mathrm{ext}}(0,T) + m \leq d-D + n'_1-1 \leq d-2 + d-2 \leq (d-2)Q'(m,T).
$$
The lemma is proved.
\end{proof}

\begin{definition}
For any $1 \leq D_1 \leq D$ and $n \geq 0$, $C_1 > 0$ we say that $T\in \Delta(\gamma)$ is $(D_1,n,C_1)$-balanced if we have
$Q_\alpha(n,T) \geq C_1^{-1}Q'(n,T)$ for at least $D_1$ of the indices $\alpha\in \mathcal A'$.
\end{definition}
Since the property only depends on the path $\gamma$ up to $n$, we will also say that this path is $(D_1, n, C_1)$-balanced. 
For instance any path is $(1, n, D)$-balanced for all $n \geq 0$.

\begin{lemma}
If $\gamma$ is $(D, N, C_0)$-balanced for some constant $C_0 > 0$ then we can 
find $\Delta'(\gamma)\subset \Delta(\gamma)$ satisfying the conclusions of the proposition
with $l = l(d)$ and $\eta = \eta(d,C_0)$.
\end{lemma}
\begin{proof}
The proof is a verbatim copy of the proof of the corresponding lemma in~\cite[Lemma 2]{MaMoYo1}.
\end{proof}

Now if $\gamma$ is only $(\widetilde{D},N,\widetilde{C})$-balanced for some $\widetilde{D}<D$, the strategy is to extend $\gamma$ without losing volume in order to obtain a more balanced path. At the end of this process, we should be able to apply above lemma and to get the conclusion of Proposition~\ref{prop:comb}. The idea is to extend the path $\gamma$ and to control the volume of the associated simplex. There are several cases to distinguish (as in~\cite{MaMoYo1}, Case A, Case B, Type I,II,III). In any case, an argument completely similar to~\cite{MaMoYo1} leads to the desired estimate.

\section{Acknowledgements} We thank G.~Forni, C.~Matheus and R.~Guti\'errez-Romo for fruitful discussions. 
The first author was partially supported by the Labex Persyval. The second author is grateful to UniCredit Bank R$\&$D group for financial support through the "Dynamics and
Information Theory Institute" at the Scuola Normale Superiore. The third author was partially supported by RFBR-CNRS grant No. 18-51--15010.

\end{document}